\documentclass[11pt]{amsart}

\usepackage{amsfonts,amssymb,amsthm,eucal,color,bbm,verbatim,graphicx,url}
\usepackage[french,english]{babel}

\usepackage[margin=1.25in]{geometry}
\calclayout

\newtheorem{thm}{Theorem}
\newtheorem{cor}[thm]{Corollary}
\newtheorem{lemma}[thm]{Lemma}
\newtheorem{prop}[thm]{Proposition}

\newtheorem{conj}[thm]{Conjecture}

\theoremstyle{remark}
\newtheorem*{rmk}{Remark}

\newcommand{\R}{\mathbb{R}}

\newcommand{\N}{\mathbb{N}}

\newcommand{\Z}{\mathbb{Z}}

\DeclareMathOperator{\conv}{conv}

\newcommand{\inprod}[2]{\left\langle #1, #2 \right\rangle}

\newcommand{\abs}[1]{\left\vert #1 \right\vert}
\newcommand{\norm}[1]{\left\Vert #1 \right\Vert}

\newcommand{\eps}{\varepsilon}

\DeclareMathOperator{\vol}{vol}

\newcommand{\Set}[2]{\left\{#1 \mathrel{} \middle| \mathrel{} #2
  \right\}}

\newcommand{\Mag}[1]{\operatorname{Mag}\left(#1\right)}
\newcommand{\lin}[1]{\operatorname{lin}(#1)}

\allowdisplaybreaks[3]

\numberwithin{thm}{section}
\numberwithin{equation}{section}

\author{Mark W.\ Meckes}

\address{Department of Mathematics, Applied Mathematics, and
  Statistics, Case Western Reserve University, 10900 Euclid Ave.,
  Cleveland, Ohio 44106, U.S.A.}

\email{mark.meckes@case.edu}

\title[Magnitude and Holmes--Thompson intrinsic
volumes]{Magnitude and Holmes--Thompson intrinsic volumes of
  convex bodies}

\begin{document}

\begin{abstract}
  Magnitude is a numerical invariant of compact metric spaces,
  originally inspired by category theory and now known to be related
  to myriad other geometric quantities. Generalizing earlier results
  in $\ell_1^n$ and Euclidean space, we prove an upper bound for the
  magnitude of a convex body in a hypermetric normed space in terms of
  its Holmes--Thompson intrinsic volumes. As applications of this
  bound, we give short new proofs of Mahler's conjecture in the case
  of a zonoid, and Sudakov's minoration inequality.
\end{abstract}

\maketitle

MSC class: 52A20; 46B20, 51F99, 52A22

\section{Introduction}
\label{S:intro}

Magnitude is a numerical isometric invariant of metric spaces recently
introduced by Leinster \cite{Leinster} based on category-theoretic
considerations.  It has rapidly found connections with a large and
growing number of areas of mathematics; see \cite{LeMe-survey} for a
survey as of 2017, sections 6.4--6.5 of \cite{Leinster-div-book} for a
more recent succinct account, and \cite{Leinster-bibliography} for a
more complete bibliography.  In appropriate contexts, magnitude
encodes a number of classical geometric quantities, including volume
\cite{Willerton-homogeneous,BaCa,LeMe-survey}, Minkowski dimension
\cite{MM-mag-etc}, surface area \cite{GiGo-AJM}, and other curvature
integrals \cite{Willerton-homogeneous,GiGo-AJM,GiGo-Willmore,GGL}.

The main result of this paper, Theorem \ref{T:hypermetric-mag},
provides an upper bound on the magnitude of a convex body in a
hypermetric normed space in terms of its Holmes--Thompson intrinsic
volumes, generalizing the main result of \cite{MM-intrinsic} for
convex bodies in Euclidean spaces. (All these terms will be defined in
the following paragraphs.) In addition to generalizing some known
results about magnitude from $\ell_1^n$ and Euclidean (or Hilbert)
spaces to more general normed spaces, we will see that this upper
bound on magnitude can be used to quickly deduce some important known
results in convex geometry, namely Mahler's conjecture in the case of
a zonoid, and Sudakov's minoration inequality.  Finally, the proof of
Theorem \ref{T:hypermetric-mag} helps elucidate the relationship
between Holmes--Thompson intrinsic volumes and Leinster's theory of
$\ell_1$ integral geometry \cite{Leinster-int}, which was developed
largely in order to state and prove the result from \cite{LeMe-survey}
on which Theorem \ref{T:hypermetric-mag} is based.

A metric space $X = (X,d)$ is called \textbf{positive definite} if
the matrix $[e^{-d(x_i, x_j)}]_{1\le i,j \le k}$ is positive definite
for every finite collection of distinct points $x_1, \dots, x_k \in
X$.  If $X$ is a compact positive definite metric space, the
\textbf{magnitude} of $X$ can be defined as
\begin{equation}
  \label{E:mag-def}
  \Mag{X,d} = \sup \Set{\frac{\mu(X)^2}{\int \int e^{-d(x,y)} \
      d\mu(x) d\mu(y)}}{\mu \in M(X)},
\end{equation}
where $M(X)$ is the set of finite signed Borel measures on $X$
\cite{MM-pdms}.  It follows immediately from this definition that
$\Mag{X,d} \in [1,\infty]$, and that magnitude is monotone with
respect to set containment for positive definite spaces.

It is a classical fact that for a normed space $E$, positive
definiteness is equivalent both to the property of being isometrically
isomorphic to a subspace of $L_1$, and being hypermetric (which for
general metric spaces is a stronger property than positive
definiteness); see, e.g., \cite[Section 6.1]{DeLa}.  We will refer to
these spaces as hypermetric normed spaces below. Examples include
$\ell_p^n = (\R^n, \norm{\cdot}_p)$ for $1 \le p \le 2$, in particular
$\ell_2^n$, which is $\R^n$ with its usual Euclidean norm.

Let $\mathcal{K}^n$ be the class of compact, convex subsets of $\R^n$,
equipped with the Hausdorff distance. Recall that a \textbf{convex
  valuation} on $\R^n$ is a function $V: \mathcal{K}^n \to \R$ such
that
\begin{equation}
  \label{E:inclusion-exclusion}
  V(K \cup L) = V(K) + V(L) - V(K\cap L)
\end{equation}
whenever $K, L, K \cup L \subseteq \mathcal{K}^n$.  A convex valuation
$V$ is said to be $m$-homogeneous for $m \in \N$ if
$V(t K) = t^m V(K)$ for every $K \in \mathcal{K}^n$ and $t>0$.

A consequence of Hadwiger's classical theorem (see e.g.\
\cite{KlRo,ScWe,Schneider}) is that up to scalar multiples, there is a
unique continuous, $m$-homogeneous, rigid motion-invariant convex
valuation $V_m$ on $\R^n$ for each $0 \le m \le n$.  With an
appropriate normalization $V_m(K) = \vol_m(K)$ whenever
$K \in \mathcal{K}^n$ is $m$-dimensional, and $V_m$ is then called the
$m^{\mathrm{th}}$ \textbf{intrinsic volume}. These quantities, under
various normalization and indexing conventions, play a central role in
integral geometry.

There are multiple natural choices for the normalization of the volume
(i.e., Lebesgue measure) on a finite-dimensional normed space $E$. For
the purposes of integral geometry, it turns out that the most
convenient normalization is the \emph{Holmes--Thompson volume} (see
e.g.\ pages 207--209 of \cite{ScWe} for discussion and references).
If $E$ is identified with $(\R^n, \norm{\cdot})$ and $\R^n$ is also
given its usual Euclidean structure, then the \textbf{Holmes--Thompson
  volume} of $A \subseteq E$ is given, up to a factor depending only
on $n$, by
\begin{equation}
  \label{E:HT-volume}
  \vol_{\mathrm{HT}}^E(A) = \vol_{2n}(A \times B^\circ),
\end{equation}
where $B$ is the unit ball of $E$,
$B^\circ = \Set{y \in \R^n}{\inprod{x}{y} \le 1 \text{ for every } x
  \in B}$ is its polar body, and $\vol_{2n}$ is the standard
symplectic volume on $\R^n \times (\R^n)^*$ (equal to the usual
normalization of Lebesgue measure on $\R^{2n}$ under the standard
identification $(\R^n)^* \cong \R^n$).  The Holmes--Thompson volume is
invariant under linear changes of coordinates, and thus independent of
the precise identification of $E$ with $\R^n$ or Euclidean structure.
If $F \subseteq E$ is an affine subspace, then we further define
$\vol_{\mathrm{HT}}^F(A) = \vol_{\mathrm{HT}}^{F_0}(A_0)$ for
$A \subseteq F$, where $F_0$ is the linear subspace of $E$ which is
parallel to $F$, and $A_0$ is a translate of $A$ lying in $F_0$.  (The
definition can be further extended to Finsler manifolds, but we will
not make use of that level of generality here.)

With the normalization defined by \eqref{E:HT-volume},
\begin{equation}
  \label{E:l2-HT-vol}
  \vol_{\mathrm{HT}}^{\ell_2^n}(A) = \vol_{2n} (A \times B_2^n) = \omega_n \vol_n(A) 
\end{equation}
and
\begin{equation}
  \label{E:l1-HT-vol}
  \vol_{\mathrm{HT}}^{\ell_1^n}(A) = \vol_{2n} (A \times B_\infty^n) = 2^n \vol_n(A). 
\end{equation}
Here and below $B_p^n$ denotes the unit ball of $\ell_p^n$ for
$1 \le p \le \infty$, and $\omega_n = \vol_n(B_2^n)$. In contrast to
this, it is standard practice to introduce a factor of $\omega_n^{-1}$
in the definition \eqref{E:HT-volume} of $\vol_{\mathrm{HT}}^E$ in
order to force the equality $\vol_{\mathrm{HT}}^{\ell_2^n} =
\vol_n$. This convention reflects the central role of the Euclidean
space $\ell_2^n$ in integral geometry.  However, in the theory of
magnitude, $\ell_1^n$ plays a more central role, and the convention
adopted above is more convenient for the statement and proof of our
main result.

The following partial analogue of Hadwiger's theorem for normed spaces
is proved in \cite{ScWi,Schneider-ivms}, although it is not given a
self-contained statement.

\begin{prop}
  \label{T:HT-unique}
  Let $E = (\R^n, \norm{\cdot})$ be a finite-dimensional hypermetric
  normed space.  For each $1 \le m \le n$ there is a unique even,
  continuous, $m$-homogeneous, translation-invariant convex valuation
  $\mu^E_m$ such that $\mu^E_m(K) = \vol_{\mathrm{HT}}^F(K)$ whenever
  $K \in \mathcal{K}^n$ is $m$-dimensional and $F \subseteq E$ is the
  affine subspace spanned by $K$.
\end{prop}

As discussed in the introduction of \cite{FaWa}, results in
\cite{AlFe,Bernig} imply that Proposition \ref{T:HT-unique} holds for
certain more general normed spaces. However, it is also shown in
\cite{Schneider-ivms} that the valuations $\mu^E_m$ necessarily have
some pathological properties (in particular, failure of monotonicity)
if $E$ is not hypermetric.

We set $\mu^E_0 = 1$. The valuations $\mu^E_m$ for $0 \le m \le n$ are
referred to as the \textbf{Holmes--Thompson intrinsic volumes} on $E$.
Note that with our normalization convention,
$\mu_m^{\ell_2^n} = \omega_m V_m$ by \eqref{E:l2-HT-vol}.  We will
denote by $\widetilde{\mu}_m^E = \omega_m^{-1} \mu_m^E$ the usual
normalization of the Holmes--Thompson intrinsic volumes, so that
$\widetilde{\mu}_m^{\ell_2^n} = V_m$.

We are now ready to state the main theorem of this paper.

\begin{thm}
  \label{T:hypermetric-mag}
  Suppose that $E = (\R^n, \norm{\cdot})$ is a hypermetric normed
  space and let $K \in \mathcal{K}^n$. Then
  \[
    \Mag{K, \norm{\cdot}} \le \sum_{m=0}^n 4^{-m} \mu_m^E(K)
    \le e^{\mu_1^E(K) / 4}.
  \]
\end{thm}

For reference, in terms of the the usual convention for
Holmes--Thompson intrinsic volumes, the conclusion of Theorem
\ref{T:hypermetric-mag} states that
\[
    \Mag{K, \norm{\cdot}} \le \sum_{m=0}^n \frac{\omega_m}{4^m} \widetilde{\mu}_m^E(K)
    \le e^{\widetilde{\mu}_1^E(K) / 2}.
\]

Theorem \ref{T:hypermetric-mag} will be deduced from \cite[Theorem
4.6]{LeMe-survey}, stated as Theorem \ref{T:l1-mag} below, which is
essentially the special case of Theorem \ref{T:hypermetric-mag} for
$E = \ell_1^n$.  The case of $E = \ell_2^n$ was previously deduced
from Theorem \ref{T:l1-mag} in \cite{MM-intrinsic}.

In Sections \ref{S:Mahler} and \ref{S:Sudakov} below we will combine
the upper bound from Theorem \ref{T:hypermetric-mag} with easy lower
bounds on magnitude to obtain new proofs of results in convex geometry
which are not obviously related to magnitude. In the rest of this
section, we will state some immediate consequences of Theorem
\ref{T:hypermetric-mag} about magnitude itself, and in particular for
the behavior of the \textbf{magnitude function}
$t \mapsto \Mag{tK, \norm{\cdot}}$ when $t \to 0$. In contrast to
this, in Section \ref{S:Mahler} we will consider the limit of the
magnitude function as $t \to \infty$, and in Section \ref{S:Sudakov}
we will use a specific finite value of $t$.  This demonstrates that
the upper bound in Theorem \ref{T:hypermetric-mag}, although not
necessarily sharp at any value of $t > 0$, is strong enough to yield
useful consequences over the entire range of rescalings of $K$.

Our first consequence of Theorem \ref{T:hypermetric-mag} is already
known.  It follows by applying the theorem to the convex hull of $X$,
using the monotonicity of magnitude.

\begin{cor}
  \label{T:fd-one-point}
  If $E = (\R^n, \norm{\cdot})$ is a finite-dimensional hypermetric
  normed space and $X \subseteq E$ is compact, then
  $\Mag{X, \norm{\cdot}} < \infty$, and
  \[
    \lim_{t\to 0^+} \Mag{tX, \norm{\cdot}} = 1.
  \]
\end{cor}

The finiteness statement in Corollary \ref{T:fd-one-point} was first
proved in this generality in \cite[Proposition 4.13]{LeMe-survey}
using Fourier analysis. In the recent paper \cite{LeMe-extremal}, a
different proof was given of the finiteness statement which also
yielded the limit statement (called the \textbf{one-point property} in
\cite{LeMe-extremal}).  Like the proof of Theorem
\ref{T:hypermetric-mag}, the proof of Corollary \ref{T:fd-one-point}
given in \cite{LeMe-extremal} is based on \cite[Theorem
4.6]{LeMe-survey}.

Theorem \ref{T:hypermetric-mag} allows Corollary \ref{T:fd-one-point}
to be generalized to certain infinite-dimensional sets.  If $K
\subseteq L_1$ is compact and convex, we define
\[
  \mu_1^{L_1}(K) = \sup\Set{\mu_1^F(K\cap F)}{F\subseteq L_1 \text{ is a
      finite-dimensional affine subspace}}.
\]
This definition is natural for subsets of $L_1$ since the
Holmes--Thompson intrinsic volumes are monotone with respect to set
containment in hypermetric normed spaces (and in fact, only in the
hypermetric case \cite{Schneider-ivms}).  Together with the facts that
magnitude is monotone and lower semicontinuous for compact positive
definite spaces \cite[Theorem 2.6]{MM-pdms}, Theorem
\ref{T:hypermetric-mag} yields the following.

\begin{cor}
  \label{T:id-one-point}
  If $X \subseteq L_1$ is compact and $\mu_1^{L_1}(\conv X) < \infty$,
  then $\Mag{X, \norm{\cdot}_1} < \infty$, and
  \[
    \lim_{t\to 0^+} \Mag{tX, \norm{\cdot}_1} = 1.
  \]
\end{cor}

A version of Corollary \ref{T:id-one-point} was proved in Corollaries
2 and 3 of \cite{MM-intrinsic} for subsets of a Hilbert space, where
the relevant class of sets are known as GB-bodies.  Since a separable
Hilbert space embeds isometrically in $L_1$, Corollary
\ref{T:id-one-point} generalizes those results.  The following
conjecture is motivated by both Corollary \ref{T:id-one-point} and the
proof of \cite[Theorem 2.1]{LeMe-extremal}, which gives the first
known example of a compact positive definite metric space with
infinite magnitude.

\begin{conj}
  Suppose that $K \subseteq L_1$ is compact and convex.  Then $\Mag{K,
    \norm{\cdot}_1} < \infty$ if and only if $\mu_1^{L_1}(K) < \infty$.
\end{conj}

The one-point property can be sharpened to the following first-order
bound on the magnitude function for small $t$, generalizing part of
\cite[Corollary 6]{MM-intrinsic} for the Euclidean case.

\begin{cor}
  \label{T:limsup}
  Suppose that $E = (\R^n, \norm{\cdot})$ is a hypermetric normed
  space and let $K \in \mathcal{K}^n$. Then
  \[
    \limsup_{t\to 0^+} \frac{\Mag{tK, \norm{\cdot}} - 1}{t} \le
    \frac{1}{4} \mu_1^E(K).
  \]
\end{cor}

The following conjecture, which generalizes \cite[Conjecture
5]{MM-intrinsic}, posits that the upper bounds in Theorem
\ref{T:hypermetric-mag} are sharp to first order for small convex
sets.

\begin{conj}
  \label{C:mu1-conj}
  Suppose that $E = (\R^n, \norm{\cdot})$ is a hypermetric normed
  space, and let $K \in \mathcal{K}^n$. Then
  \[
    \lim_{t \to 0^+} \frac{\Mag{tK, \norm{\cdot}} - 1}{t} = \frac{1}{4} \mu_1^E(K).
  \]
\end{conj}

When $E = \ell_1^n$, Conjecture \ref{C:mu1-conj} holds whenever $K$
has nonempty interior, by Theorem \ref{T:l1-mag} below.  When
$E = \ell_2^n$ the conjecture holds if $n$ is odd and $K = B_2^n$, by
\cite[Theorem 4]{MM-intrinsic}.  In all other cases the conjecture is
open, although the results of \cite{GiGo-AJM} imply that if
$E = \ell_2^n$, $n$ is odd, and $K$ has smooth boundary, then the
limit exists.

\medskip

The rest of this paper is organized as follows.  In Section
\ref{S:proof} we will prove Theorem \ref{T:hypermetric-mag}.  In
Sections \ref{S:Mahler} and \ref{S:Sudakov} we will see how Theorem
\ref{T:hypermetric-mag} quickly yields, respectively, Mahler's
conjecture for zonoids and Sudakov's minoration inequality.  Finally,
in Section \ref{S:integral} we will make some remarks about Leinster's
$\ell_1$ integral geometry, which underlies Theorem \ref{T:l1-mag},
and its relationships to both the theory of Holmes--Thompson intrinsic
volumes and the Wills functional.

\section{Proof of Theorem \ref{T:hypermetric-mag}}
\label{S:proof}

To state the theorem on which the proof of Theorem
\ref{T:hypermetric-mag} is based, we first define some additional
notation. Following \cite{Leinster-int}, for $0 \le m \le n$, we
define the \textbf{$\ell_1$ intrinsic volumes} of
$K \in \mathcal{K}^n$ by
\[
V_m'(K) = \sum_{P \in \mathrm{Gr}'_{n,m}} \vol_m(K\vert P),
\]
where $\mathrm{Gr}'_{n,m}$ denotes the set of $m$-dimensional
coordinate subspaces of $\R^n$ and $K \vert P$ denotes the orthogonal
projection of $K$ onto $P$.  (The natural class of
sets to consider here is actually larger than convex bodies, a point
that we will return to in Section \ref{S:integral}.) Note that if $K$
lies in a $d$-dimensional subspace of $\ell_1^n$, then $V_m'(K) = 0$
for $m > d$.  It follows that
\[
V_m'(K) = \frac{1}{m!} \sum_{i_1, \dots, i_m = 1}^n \vol_m
\bigl(P_{i_1, \dots, i_m} (K)\bigr),
\]
where $P_{i_1, \dots, i_m}:\R^n \to \R^m$ is the linear map represented by
the matrix whose rows are the standard basis vectors $e_{i_1}, \dots,
e_{i_m} \in \R^n$.

The following result is part of Theorem 4.6 of \cite{LeMe-survey}.

\begin{thm}
  \label{T:l1-mag}
  If $K \in \mathcal{K}^n$, then
  \[
    \Mag{K, \norm{\cdot}_1} \le \sum_{m=0}^n 2^{-m} V_m'(K),
  \]
  with equality when $K$ has nonempty interior.
\end{thm}

To deduce Theorem \ref{T:hypermetric-mag} from Theorem \ref{T:l1-mag},
we approximate a hypermetric normed space $E$ by a sequence of
$n$-dimensional subspaces $E_N \subseteq \ell_1^N$.  To do this we use
the following fact, which goes back to L\'evy; see e.g.\ \cite[Section
6.1]{Koldobsky}.

\begin{prop}
  \label{T:generating-measure}
  A finite-dimensional normed space $E = (\R^n, \norm{\cdot})$ is
  hypermetric if and only if there exists an even nonnegative measure
  $\rho$ on $S^{n-1}$ such that
  \[
    \norm{x} = \int \abs{\inprod{x}{y}} \ d\rho(y)
  \]
  for all $x \in \R^n$.
\end{prop}

From the perspective of convex geometry, Proposition
\ref{T:generating-measure} implies that $E$ is hypermetric if and only
if $B^\circ$, the polar body of the unit ball of $E$, is a zonoid (see
\cite[Theorem 3.5.3]{Schneider}). In that setting $\rho$ is referred
to as the \textbf{generating measure} of $B^\circ$; we will also refer
to it as the generating measure for $E$. In \cite{ScWi}, Schneider and
Wieacker investigated Holmes--Thompson intrinsic volumes for
hypermetric normed spaces with the help of generating measures.  We
will use the following expression they derived (see \cite[formula
(64)]{ScWi}).

\begin{prop}
  \label{T:L1-HT-iv}
  Suppose that $E = (\R^n, \norm{\cdot})$ is a hypermetric normed
  space with generating measure $\rho$. Then
  \begin{align*}
    \mu_m^E(K) & = c_m \int_{(S^{n-1})^m} \vol_m \bigl(K \vert \lin{x_1,
                 \dots, x_m}\bigr) \sqrt{\det (A A^t)} \ d\rho(x_1) \cdots
                 d\rho(x_m) \\
               & = c_m \int_{(S^{n-1})^m} \vol_m (A K) \ d\rho(x_1) \cdots
                 d\rho(x_m)
  \end{align*}
  where $c_m$ depends only on $m$.  Here $\lin{x_1,\dots x_n}$ denotes
  the linear span of $x_1,\dots, x_m \in \R^n$ and $A$ is the $m \times n$ matrix
  with rows $x_1, \dots, x_m$.
\end{prop}

Since we are using a different normalization convention than
\cite{ScWi}, the value of $c_m$ here differs from the one stated
in \cite{ScWi}. The proof of the following corollary shows that for
our normalization, $c_m = \frac{2^m}{m!}$.

\begin{cor}
  \label{T:l1-HT-iv}
  For $K \in \mathcal{K}^n$, \( \mu_m^{\ell_1^n}(K) = 2^m V_m'(K). \)
\end{cor}

\begin{proof}
  The generating measure for $\ell_1^n$ is
  $\rho = \frac{1}{2} \sum_{i=1}^n (\delta_{e_i} + \delta_{-e_i})$. If
  $x_1, \dots, x_m \in \{\pm e_1, \dots, \pm e_n\}$, then
  $\sqrt{\det (A A^t)} = 1$ if $\lin{x_1, \dots, x_m}$ is
  $m$-dimensional, and is $0$ otherwise. Proposition \ref{T:L1-HT-iv}
  therefore implies that
  \begin{align*}
    \mu_m^{\ell_1^n}(K) & = c_m \int_{(S^{n-1})^m} \vol_m \bigl(K \vert \lin{x_1, \dots,
    x_m}\bigr) \sqrt{\det (A A^t)} \ d\rho(x_1) \cdots d\rho(x_m) \\
    & = 2^{-m} c_m \sum_{j_1, \dots, j_m = 1}^n
      \sum_{\eps_1, \dots, \eps_m \in \{1, -1\}} \vol_m \bigl(K \vert
      \lin{\eps_1 e_{j_1}, \dots, \eps_m e_{j_m}}\bigr) \\
    & = c_m \sum_{j_1, \dots, j_m = 1}^n
      \vol_m \bigl(K \vert
      \lin{e_{j_1}, \dots, e_{j_m}}\bigr) \\
    & = m! c_m V_m'(K).
  \end{align*} 
  Now if $E \subseteq \ell_1^n$ is an $m$-dimensional coordinate
  subspace, then $E$ is isometrically isomorphic to $\ell_1^m$, and
  when $K \subseteq E$ we have
  \[
    \mu_m^{\ell_1^n}(K) = \vol_{\mathrm{HT}}^E(K) = 2^m \vol_m(K) =
    2^m V_m'(K)
  \]
  by \eqref{E:l1-HT-vol} and \cite[Lemma 5.2]{Leinster-int}, and so
  $c_m = \frac{2^m}{m!}$.
\end{proof}

Corollary \ref{T:l1-HT-iv} shows in particular that when
$E = \ell_1^n$, the first inequality in Theorem
\ref{T:hypermetric-mag} reduces to Theorem \ref{T:l1-mag}.  It also
implies the following generalization of \cite[Lemma
5.2]{Leinster-int}.

\begin{cor}
  \label{T:Vi'-subspace}
  If $K \in \mathcal{K}^n$ lies in an $m$-dimensional subspace $E
  \subseteq \R^n$, then
  \[
    V_m'(K) = \frac{\vol_m\bigl( (B_1^n \cap E)^\circ\bigr)}{2^m} \vol_m(K),
  \]
  where the polar body $(B_1^n \cap E)^\circ$ is considered in the
  subspace $E$.
\end{cor}

For the second inequality in Theorem \ref{T:hypermetric-mag} we will
need the following generalization of \cite[Lemma 3.2]{LeMe-extremal}.

\begin{lemma}
  \label{T:mu1-ub}
  If $E = (\R^n, \norm{\cdot})$ is a hypermetric normed space and
  $K \in \mathcal{K}^n$, then
  $\mu_{i+j}^E(K) \le \frac{i! j!}{(i+j)!} \mu_i^E(K) \mu_j^E(K)$ for
  all $i, j \ge 0$. Consequently
  $\mu_m^E(K) \le \frac{1}{m!} (\mu_1^E(K))^m$ for each $1\le m \le n$.
\end{lemma}

\begin{proof}
  Let $x_1, \dots, x_{i+j} \in \R^n$.  Writing $A$ for the matrix with
  rows $x_1, \dots, x_{i+j}$, $A_1$ for the matrix with rows
  $x_1, \dots, x_i$, and $A_2$ for the matrix with rows
  $x_{i+1}, \dots, x_j$, we have $A K \subseteq A_1 K \times A_2
  K$. Proposition \ref{T:L1-HT-iv} then implies that
  \begin{align*}
    \mu_{i+j}^E(K) & 
    \le \frac{2^{i+j}}{(i+j)!} \int_{(S^{n-1})^{i+j}} \vol_i (A_1 K)
      \vol_j(A_2 K) \ d\rho(x_1) \cdots d\rho(x_{i+j}) \\
    & = \frac{i! j!}{(i+j)^!} \mu_i^E(K) \mu_j^E(K).
  \end{align*}
  This implies in particular that
  $\mu_{j+1}^E(K) \le \frac{1}{j+1} \mu_1^E(K) \mu_j^E(K)$ for each $j$, and the
  second claim now follows by induction.
\end{proof}

We are now ready to prove the main result.

\begin{proof}[Proof of Theorem \ref{T:hypermetric-mag}]
  Let $\rho$ be the generating measure for $E$.  We can
  approximate $\rho$ by a sequence of discrete measures
\[
\rho_N = \sum_{j=1}^{N} w_{N,j} \delta_{\theta_{N,j}}
\]
with $w_{N,j} > 0$ and $\theta_{N,j} \in S^{n-1}$.  For each $N$ we
get a seminorm, which for sufficiently large $N \ge n$ will be a norm,
given by 
\[
  \norm{x}_{E_N} = \int_{S^{n-1}} \abs{\inprod{x}{y}} \ d\rho_N(y) =
  \sum_{j=1}^m w_{N,j} \abs{\inprod{x}{\theta_{N,j}}} = \sum_{j=1}^N
  \abs{\inprod{x}{w_{N,j} \theta_{N,j}}}.
\]
We write $E_N = (\R^n, \norm{\cdot}_{E_N})$.  Define $T_N:\R^N \to \R^n$
by $T_N(e_j) = w_{N,j} \theta_{N,j}$.  Then
\[
  \norm{T^*_N(x)}_1 = \sum_{j=1}^N \abs{\inprod{T_N^*(x)}{e_j}}
  = \sum_{j=1}^N \abs{\inprod{x}{T_N(e_j)}} = \norm{x}_{E_N},
\]
so $T_N^*$ is an isometric embedding of $E_N$ into $\ell_1^N$.

We have
\[
  \norm{x}_{E_N} = \int_{S^{n-1}} \abs{\inprod{x}{y}} \ d\rho_N(y)
  \xrightarrow{N \to \infty} \int_{S^{n-1}} \abs{\inprod{x}{y}} \
  d\rho(y) = \norm{x}.
\]
This implies that $(K, \norm{\cdot}_{E_N}) \to (K, \norm{\cdot})$ in
the Gromov--Hausdorff metric (see \cite[Section 3.A]{Gromov}). Since
magnitude is lower semicontinuous with respect to the
Gromov--Hausdorff metric on the class of compact positive definite
metric spaces \cite[Theorem 2.6]{MM-pdms}, it follows that
\begin{equation}
  \label{E:liminf}
  \Mag{K, \norm{\cdot}} \le \liminf_{N \to \infty} \Mag{K, \norm{\cdot}_{E_N}} = \liminf_{N \to
    \infty} \Mag{T_N^* (K), \norm{\cdot}_1}.
\end{equation}

Now
\begin{equation}
  \label{E:Vi'-limit}
  \begin{split}
  V_m'(T_N^*(K))
  & = \frac{1}{m!} \sum_{i_1, \dots, i_m = 1}^N
    \vol_m \bigl(P_{i_1, \dots, i_m} T_N^*(K) \bigr) \\
  & = \frac{1}{m!} \sum_{i_1, \dots, i_m = 1}^N w_{N,i_1} \cdots w_{N,i_N}
    \vol_m \left( \begin{bmatrix} \theta_{N,i_1} \\ \vdots
        \\ \theta_{N,i_m} \end{bmatrix} K \right) \\
  & = \frac{1}{m!} \int \dots \int \vol_m
    (A K ) \ d\rho_N(x_1) \cdots d\rho_N(x_m) \\
  & \xrightarrow{N \to \infty}
    \frac{1}{m!} \int \dots \int \vol_m (A K )
  \ d\rho(x_1) \cdots d\rho(x_m) \\
  & = 2^{-m} \mu_m^{E}(K),
  \end{split}
\end{equation}
where the last equality follows from Proposition \ref{T:L1-HT-iv}.
Here $A$ as before stands for the matrix with rows $x_1, \dots, x_m$,
and the matrix in the second row has rows
$\theta_{N,i_1}, \dots, \theta_{N,i_m}$.  The first inequality in
Theorem \ref{T:hypermetric-mag} now follows by combining
\eqref{E:liminf}, Theorem \ref{T:l1-mag}, and \eqref{E:Vi'-limit}. The
second inequality then follows by Lemma \ref{T:mu1-ub}.
\end{proof}

\section{Application: Mahler's conjecture for zonoids}
\label{S:Mahler}

In this and the following section, we will see how two important
results in convex geometry which make no reference to magnitude
quickly follow by combining the upper bound on magnitude from Theorem
\ref{T:hypermetric-mag} with easy lower bounds.

Mahler conjectured in 1939 \cite{Mahler} that if $K \in \mathcal{K}^n$
is symmetric with nonempty interior, then
\[
  \vol_n(K) \vol_n(K^\circ) \ge \frac{4^n}{n!}.
\]
Equality is attained (nonuniquely) for $K = B_1^n$ or
$K = B_\infty^n$. This has been proved in various special cases and in
asymptotic forms (see \cite{IrSh} for a proof when $n=3$ and further
references) but the general case remains open.

In the proof of the following result, we compare the top-order
behavior of the first upper bound on magnitude in Theorem
\ref{T:hypermetric-mag}, which is typically not sharp for large convex
bodies, with a lower bound that is known to be asymptotically sharp.
This comparison immediately implies Mahler's conjecture for zonoids,
first proved in \cite{Reisner-rpvp,Reisner-zonoids} (see also
\cite{GoMeRe}).

\begin{cor}
  \label{T:Mahler-zonoid}
  If $Z \in \mathcal{K}^n$ is an $n$-dimensional zonoid, then 
  \[
    \vol_n(Z) \vol_n(Z^\circ) \ge \frac{4^n}{n!}.
  \]
\end{cor}

\begin{proof}
  Let $E = (\R^n, \norm{\cdot})$ be the hypermetric normed space with
  unit ball $B = Z^\circ$.  Theorem
  \ref{T:hypermetric-mag} implies that for $t \to \infty$,
  \[
    \Mag{t B,\norm{\cdot}} \le 4^{-n} \mu_n^E(B) t^n + O(t^{n-1}) =
    4^{-n} \vol_n(B) \vol_n(B^\circ) t^n + O(t^{n-1}).
  \]
  On the other hand, for each $t > 0$ we have the lower bound
  \[
    \Mag{t B, \norm{\cdot}} \ge \frac{t^n}{n!}
  \]
  (see \cite[Theorem 3.5.6]{Leinster} or \cite[Proposition
  4.13]{LeMe-survey}). Combining these, we obtain
  \[
    \frac{4^n}{n!} \le \vol_n(B) \vol_n(B^\circ) + O(t^{-1}) =
    \vol_n(Z) \vol_n(Z^\circ) + O(t^{-1}),
  \]
  and letting $t \to \infty$ proves the claim.
\end{proof}

\begin{rmk}
  Although we considered $\Mag{tB, \norm{\cdot}}$ in the above proof
  for convenience, we could equally well consider
  $\Mag{tK, \norm{\cdot}}$ for any $n$-dimensional convex body $K$
  (with the norm still corresponding to $B = Z^\circ$) and obtain the
  same result.
\end{rmk}

\section{Application: Sudakov minoration}
\label{S:Sudakov}

Our last application of Theorem \ref{T:hypermetric-mag} is a new proof
of Sudakov's minoration inequality, which is a key tool in both
high-dimensional convex geometry and the theory of Gaussian processes
(see e.g.\ \cite{AGM} and \cite{Talagrand-book}, respectively).  This
application uses only the special case of Theorem
\ref{T:hypermetric-mag} when $E = \ell_2^n$, proved previously in
\cite{MM-intrinsic}.  In that case, the first upper bound in Theorem
\ref{T:hypermetric-mag} can be combined with a stronger counterpart to
Lemma \ref{T:mu1-ub} in Euclidean space to deduce the following
sharper version of the second upper bound.

\begin{cor}
  \label{T:mag-V1}
  If $K \in \mathcal{K}^n$, then
  \[
    \Mag{K, \norm{\cdot}_2} \le e^{C V_1(K)^{2/3}}.
  \]  
\end{cor}

Throughout this section, $c$, $C$, and $C'$ refer to absolute positive
constants whose values may differ from one instance to another.

\begin{rmk}
  Corollary \ref{T:mag-V1} can be extended to infinite-dimensional
  Hilbert spaces, similarly to Corollary \ref{T:id-one-point}, but for
  simplicity we will restrict attention to finite dimensions in this
  section.
\end{rmk}

\begin{proof}[Proof of Corollary \ref{T:mag-V1}]
  It was independently shown by Chevet \cite[Lemme 4.2]{Chevet} and McMullen
  \cite[Theorem 2]{McMullen-ineq} that the Alexandrov--Fenchel inequalities imply that
  $V_m \le \frac{1}{m!} V_1^m$ for every $m \ge 1$. As observed in
  formula (17) of \cite{MM-intrinsic}, Theorem \ref{T:hypermetric-mag}
  then implies that
  \[
    \Mag{K, \norm{\cdot}_2} \le
    \sum_{m=0}^n \frac{\omega_m}{m!}  \left(\frac{V_1(K)}{4}\right)^m
    = \sum_{m=0}^n \frac{1}{\Gamma \bigl(1 + \frac{m}{2}\bigr) m!}
    \left(\frac{\sqrt{\pi} V_1(K)}{4}\right)^m.
  \]
  
  We now consider the function
  \[
    f(x) = \sum_{m=0}^\infty \frac{x^m}{\Gamma \bigl(1 + \frac{m}{2}\bigr) m!},
  \]
  a special case of Wright's generalized hypergeometric function.  We
  claim that $f(x) \le e^{c x^{2/3}}$ for $x > 0$, which will complete
  the proof.

  If $x \le 1$, then since $\Gamma \bigl(1 + \frac{m}{2}\bigr) \ge
  \frac{\sqrt{\pi}}{2}$ for every $m \ge 0$, we have
  \[
    f(x) \le \exp\left(\frac{2}{\sqrt{\pi}} x\right) \le
    \exp\left(\frac{2}{\sqrt{\pi}} x^{2/3}\right).
  \]
  For $x \ge 1$, Stirling's approximation implies that
  \begin{align*}
    f(x) & = 1 + \sum_{k=1}^\infty \frac{x^{2k}}{(2k)! k!}
           + \sum_{k=1}^\infty \frac{x^{2k-1}}{(2k-1)! \Gamma(1 +
           \frac{2k-1}{2})}
           \le 1 + \sum_{k=1}^\infty \frac{c^k x^{2k}}{(3k)!} \\
         & = 1 + \sum_{k=1}^\infty \frac{(c^{1/3} x^{2/3})^{3k}}{(3k)!}
           \le e^{c^{1/3} x^{2/3}}.   
           \qedhere
  \end{align*}
\end{proof}

Since $V_1$ is $1$-homogeneous, Corollary \ref{T:mag-V1} is equivalent
to the following.

\begin{cor}
  \label{T:mag-V1-t}
  If $K \in \mathcal{K}^n$, then
  \[
    V_1(K) \ge C \sup_{t > 0} t^{-1} \bigl(\log \Mag{t K, \norm{\cdot}_2}\bigr)^{3/2}.
  \]
\end{cor}

\begin{cor}[Sudakov's minoration inequality]
  \label{T:Sudakov}
  Let $K \in \mathcal{K}^n$, and suppose there exist
  $x_1, \dots, x_N \in K$ such that $\norm{x_i - x_j}_2 \ge \eps$
  whenever $i\neq j$. Then
  \[
    V_1(K) \ge C \eps \sqrt{\log N}.
  \]
\end{cor}

\begin{proof}
  Assume without loss of generality that $N \ge 2$, and let $\mu =
  \sum_{i=1}^N \delta_{x_i}$. By \eqref{E:mag-def},
  \[
    \Mag{t K, \norm{\cdot}_2} \ge \frac{N^2}{\int_K \int_K e^{-t
        \norm{x-y}} \ d\mu(x) d\mu(y)}
    \ge \frac{N}{1 + N e^{-t \eps}}.
  \]
  If $t = \frac{\log (2N)}{\eps}$ this implies that
  $\Mag{t K, \norm{\cdot}_2} \ge \frac{2}{3} N$, and so Corollary
  \ref{T:mag-V1-t} implies that
  \[
    V_1(K) \ge C \eps \frac{\left[\log \left(\frac{2}{3} N
        \right)\right]^{3/2}}{\log (2N)}
    \ge C' \eps \sqrt{\log N}.
    \qedhere
  \]
\end{proof}

\begin{rmk}
  If the supremum in our definition \eqref{E:mag-def} of magnitude is
  restricted to positive measures $\mu$, we obtain a quantity called
  the \textbf{maximum diversity} of $(X,d)$, denoted $D_{\max}(X,d)$
  (see \cite{MM-pdms,LeRo}). The above proof of Corollary
  \ref{T:Sudakov} shows that
  \[
    \sup_{t > 0} t^{-1} \bigl(\log D_{\max}(t K,
    \norm{\cdot}_2)\bigr)^{3/2} \ge C \sup_{\eps > 0} \eps \sqrt{\log
      N(K,\eps)},
  \]
  where $N(K,\eps)$ is the maximum size of a collection of
  $\eps$-separated points in $K$. It can similarly be shown that 
  \[
    \sup_{t > 0} t^{-1} \bigl(\log D_{\max}(t K,
      \norm{\cdot}_2)\bigr)^{3/2} \le C' \sup_{\eps > 0} \eps
    \sqrt{\log N(K,\eps)}
  \]
  (cf.\ the proof of \cite[Theorem 7.1]{MM-mag-etc}).  It follows that
  Sudakov's minoration inequality is equivalent, up to the value of the
  constant $C$, to the inequality
  \[
        D_{\max}(K, \norm{\cdot}_2) \le e^{C V_1(K)^{2/3}},
  \]
  a weaker counterpart of Corollary \ref{T:mag-V1}. 

  This observation suggests trying to prove sharper lower bounds on
  $V_1(K)$ than provided by Sudakov's inequality by using Corollary
  \ref{T:mag-V1} and bounding $\Mag{K, \norm{\cdot}_2}$ from below by
  leveraging the fact that the supremum in \eqref{E:mag-def} is over a
  space of \emph{signed} measures. We recall that optimal lower bounds
  on $V_1(K)$ are given by Talagrand's celebrated majorizing measure
  theorem \cite{Talagrand-rgp} and its more recent reformulations
  \cite{Talagrand-book}, but those bounds are not easy to apply in
  practice (see e.g.\ \cite{vanHandel} for discussion of this). In
  general the supremum in \eqref{E:mag-def} is not achieved even in
  the space of signed measures, but the definition of magnitude can be
  reformulated in several ways that invite consideration from the
  perspective of distributions and partial differential equations
  \cite{MM-mag-etc}. This perspective has led to the sharpest known
  results on magnitude in Euclidean spaces
  \cite{BaCa,Willerton-ball1,Willerton-ball2,GiGo-AJM,GiGo-Willmore,GGL},
  and may be similarly fruitful in this setting.
\end{rmk}

\section{Some remarks on $\ell_1$ integral geometry}
\label{S:integral}

The Holmes--Thompson intrinsic volumes were introduced in order to
find natural generalization of results from integral geometry in
Euclidean spaces to more general normed spaces (or still more
generally, Finsler manifolds).  In particular, Schneider and Wieacker
\cite{ScWi} showed that in hypermetric normed spaces, the
Holmes--Thompson intrinsic volumes satisfy versions of the classical
Crofton formula; see \cite{AlFe,Bernig} for versions in more general
settings.

In \cite{Leinster-int}, Leinster similarly proved a suite of results
involving the $\ell_1$ intrinsic volumes which are counterparts of
classical integral geometric theorems. Since Corollary
\ref{T:l1-HT-iv} shows that up to scaling, Leinster's $\ell_1$
intrinsic volumes $V_m'$ applied to convex sets are precisely the
Holmes--Thompson intrinsic volumes for $\ell_1^n$, one might guess
that Leinster's theory is subsumed by Holmes--Thompson integral
geometry. However, there are at least two major parts of Euclidean
integral geometry for which Leinster proved $\ell_1$ analogues in
\cite{Leinster-int}, but for which no general Holmes--Thompson version
is known.

The first is Hadwiger's theorem (see, e.g.\
\cite{KlRo,ScWe,Schneider}), which states that every continuous, rigid
motion-invariant convex valuation on $\ell_2^n$ is a linear
combination of the Euclidean intrinsic volumes.  In general normed
spaces, Proposition \ref{T:HT-unique} classifies only homogeneous
valuations with a normalization condition that serves as a proxy for
rigid motion-invariance, whereas Hadwiger's theorem also implies that
invariant convex valuations are linear combinations of these
homogeneous valuations.  In $\ell_1^n$, Leinster proved an exact
analogue of Hadwiger's theorem assuming invariance only under the
isometry group for the $\ell_1^n$ norm \cite[Theorem
5.4]{Leinster-int}. To compensate for this smaller isometry group,
Leinster assumes the valuations are defined and satisfy
\eqref{E:inclusion-exclusion} on the larger class of
\textbf{$\ell_1$-convex} sets, i.e., sets that are geodesic with
respect to the $\ell_1^n$ metric. (Indeed, the fact that Leinster's
$\ell_1$ intrinsic volumes satisfy \eqref{E:inclusion-exclusion} for
all $\ell_1$-convex sets is crucial to the proof of Theorem
\ref{T:l1-mag}, even when that theorem is restricted to convex sets.)
This suggests the possibility of stronger Hadwiger-like theorems in
normed spaces than Proposition \ref{T:HT-unique} for valuations with
suitably chosen domains. As discussed in \cite{Leinster-int}, however,
the most naive generalization of the $\ell_1$ and Euclidean versions
of Hadwiger's theorem is typically false.

Second, in \cite[Theorem 6.2]{Leinster-int} Leinster proved the
following $\ell_1$ version of Steiner's formula (see e.g.\
\cite[equation (4.1)]{Schneider}): if $X \subseteq \R^n$ is
$\ell_1$-convex then
\begin{equation}
  \label{E:l1-Steiner}
  \vol_n\bigl(X + t[0,1]^n\bigr) = \sum_{m=0}^n V_m'(X) t^{n-m}.
\end{equation}
This formula implies in particular that the $\ell_1$ intrinsic
volumes, like the Euclidean intrinsic volumes, are particular
instances of mixed volumes \cite[Section 5.1]{Schneider}.
Holmes--Thompson intrinsic volumes are not known to have
representations as mixed volumes in general; furthermore, a
Steiner-like formula such as \eqref{E:l1-Steiner}, which would require
the intrinsic volumes on the right hand side to be mixed volumes of a
particularly simple form, can only hold under additional restrictions
on the normed space $E$.  See \cite[Section 5]{Schneider-ivms} for
some partial results and discussion of these issues.

We end with a simple observation related to \eqref{E:l1-Steiner}.  As
noted in \cite{LeMe-extremal}, the quantity
\[
  \mathcal{W}'(X) = \sum_{m=0}^n V_m'(X)
\]
is an $\ell_1$ analogue of the \textbf{Wills functional} (see e.g.\
\cite{AlHeYe}), which can be defined by
\[
\mathcal{W}(K) = \sum_{m=0}^n V_m(K).
\]
The Wills functional was introduced in \cite{Wills}, where it was
conjectured that
\[
\# (K \cap \Z^n) \le \mathcal{W}(K)
\]
for any $K \in \mathcal{K}^n$.  This was shown by Hadwiger
\cite{Hadwiger-Wills} to be false for sufficiently large $n$.
However, \eqref{E:l1-Steiner} implies that an $\ell_1$ version of this
conjecture is true in all dimensions.

\begin{prop}
  \label{T:l1-Wills}
  Suppose that $X \subseteq \R^n$ is compact and $\ell_1$-convex. Then
  \[
    \# (X \cap \Z^n) \le \mathcal{W}'(X).
  \]
\end{prop}

\begin{proof} 
  By \eqref{E:l1-Steiner},
  \[
    \# (X \cap \Z^n) = \vol_n\bigl((X\cap \Z^n) +
    [0,1]^n\bigr) \le \vol_n\bigl(X + [0,1]^n\bigr) = \mathcal{W}'(X).
    \qedhere
  \]
\end{proof}

\section*{Acknowledgements}
This work was supported in part by Collaboration Grant 315593 from the
Simons Foundation. This work was partly done while the author was
visiting the Mathematical Institute of the University of Oxford, with
support from ERC Advanced Grant 740900 (LogCorRM) to Prof.\ Jon
Keating and Simons Fellowship 678148 to Elizabeth Meckes. The author
thanks the Institute and Prof.\ Keating for their hospitality, Thomas
Wannerer for helpful comments and pointers to the literature on
Holmes--Thompson intrinsic volumes, and an anonymous referee for
helpful comments on the exposition.

\bibliographystyle{plain}
\bibliography{mag-HT}

\end{document}